\documentclass{birkjour}
 \newtheorem{thm}{Theorem}[section]
 \newtheorem{cor}[thm]{Corollary}
 \newtheorem{lem}[thm]{Lemma}
 
 \theoremstyle{definition}
 \newtheorem{defn}[thm]{Definition}
 \newtheorem{ex}[thm]{Example}
 \theoremstyle{remark}
 \newtheorem{rem}[thm]{Remark}
 \numberwithin{equation}{section}
\usepackage{color}
\begin{document}

%
%
%
%
%
%
%
%
%

\title[Matrix methods for perfect signal  recovery]
 {Matrix Methods for Perfect Signal Recovery   Underlying Range Space of   Operators}



\author[F. Arabyani Neyshaburi]{F. Arabyani-Neyshaburi}
\address{Department of Mathematics,  Ferdowsi University of Mashhad  and Center of Excellence in Analysis on Algebraic Structures (CEAAS), P.O.Box 1159-91775, Mashhad, Iran.}
\email{fahimeh.arabyani@gmail.com}

\author[R. A. Kamyabi-Gol]{R. A.  Kamyabi-Gol}
\address{Department of Mathematics, Faculty of Math, Ferdowsi University of Mashhad  and Center of Excellence in Analysis on Algebraic Structures (CEAAS), P.O.Box 1159-91775, Mashhad, Iran.}
\email{kamyabi@um.ac.ir}


\vspace{1.7cm}
\begin{abstract}
The most important purpose of this article is to investigate perfect reconstruction underlying range space of  operators in finite dimensional Hilbert spaces by matrix methods. To this end, first we obtain more structures of the canonical $K$-dual.
Then, we survey  the problem of recovering and robustness  of signals  when the erasure set satisfies  the  minimal redundancy condition or the  $K$-frame is maximal robust.   Furthermore, we show that  the error rate is reduced under erasures if  the $K$-frame is of uniform excess. Toward the protection of encoding frame (K-dual) against erasures, we introduce a new  concept so called   $(r,k)$-matrix to recover lost data and solve the  perfect recovery problem  via matrix equations. Moreover, we   discuss  the existence of such matrices by using  minimal redundancy condition on decoding frames for operators. Finally, we  exhibit  several examples that illustrate our results and the advantage of using   the new  matrix with respect to previous approaches in existence and construction.
\end{abstract}

\maketitle

\maketitle
\section{Introduction and preliminaries}

\smallskip
\goodbreak
The theory of frames has  established  efficient algorithms for a wide range of  applications  in the last twenty years \cite{ Ben06, Bod05, Bol98, Donoho, larson}.
In most of  those applications, they deal with  dual frames to reconstruct the modified data and compare it with the original data.
In frame theory setting, an original signal $f$  is encoded  by the measurements $\theta_{F}^{*}f$ (encoded coefficients), where $\theta_{F}^{*}$ is the analysis operator of a frame $F$. Then,
from these measurements $f$ can be recovered  applying a reconstruction formula
by a dual frame $G$ (decoding frame) as $\theta_{G}\theta_{F}^{*}f$.  In real applications, in these transmissions, usually  a part of
the data vectors are corrupted or lost, and we may have to perform the reconstruction by using
the partial information at hand. So, searching for the  best
dual frames that minimize  the reconstruction errors when erasures occur,  optimal dual problem, is one of the  most important problems in frame theory that was   introduced by Han et al. in \cite{leng, lopez}. To state the  optimal dual problem, we first recall some basic notations of finite  classical frames. Let $\mathcal{H}_{n}$ be an $n$-dimensional   Hilbert space and $I_{m}=\{1, 2, ... ,m\}$. A sequence $F:=\{f_{i}\}_{i\in I_{m}}\subseteq\mathcal{H}_{n}$ is called a
\textit{frame} for $\mathcal{H}_{n}$ whenever ${\textit{span}}\{f_{i}\}_{i\in I_{m}}=\mathcal{H}_{n}$. The \textit{synthesis operator} $\theta_{F}:
 l^{2}(I_{m})\rightarrow\mathcal{H}_{n}$ is defined by $\theta_{F}\lbrace c_{i}\rbrace= \sum_{i\in I_{m}}
c_{i}f_{i}$. If $\lbrace f_{i}\rbrace_{i\in I_{m}}$ is a frame, then $S_{F}=
\theta_{F}\theta_{F}^{*}$ is called frame operator where $\theta_{F}^{*}: \mathcal{H}_{n}\rightarrow l^{2}(I_{m})$, the adjoint
of $\theta_{F}$, is given by $\theta_{F}^{*}f= \lbrace \langle f,f_{i}\rangle\rbrace_{i
\in I_{m}}$ and is known as the \textit{analysis operator}. A sequence $G:=\{g_i\}_{i\in I_{m}}\subseteq\mathcal{H}_{n}$ is called a
\textit{ dual}
 for $\{f_{i}\}_{i\in I_{m}}$ if $\theta_{G}\theta_{F}^{*}=I_{\mathcal{H}_{n}}$.
 A special dual frame as $\{S_{F}^{-1}f_{i}\}_{i\in I_{m}}$  is
called the canonical dual of $F$. It is well known that $\{g_i\}_{i\in I_{m}}$ is a dual  frame of $\{f_i\}_{i\in I_{m}}$ if and only if $g_{i} = S_{F}^{-1}f_{i}+u_{i}$, for all $i\in I_{m}$ where $U=\{u_i\}_{i\in I_{m}}$ satisfies $\theta_{F}\theta^{*}_{U}=0$. We refer the reader to  \cite{Chr08} for more information on frame theory. The optimal dual problem,  imposes the following problem:
let $F = \{f_{i}\}_{i\in I_{m}}$ be a frame for
$\mathcal{H}_{n}$ (encoding frame), find a dual frame of $F$ that minimize  the reconstruction errors when erasures occur. If $G =
\{g_{i}\}_{i\in I_{m}}$ is a dual of $F$ (decoding frame) and $\Lambda \subset I_{m}$, then
the error operator $E_{\Lambda}$ is defined by
\begin{eqnarray*}
E_{\Lambda} = \sum_{i\in \Lambda}f_{i} \otimes g_{i} = \theta_{G}^{*}D\theta_{F},
\end{eqnarray*}
where $D$ is an  $m\times m$ diagonal
matrix with $d_{ii}=1$ for $i\in \Lambda$ and $0$ otherwise.    Let
\begin{eqnarray}\label{01.optimal}
 d_{r}(F,G) = \max \{ \Vert \theta_{G}^{*}D\theta_{F}\Vert : D\in \mathcal{D}_{r}\} = \max\{\|E_{\Lambda}\| : \vert \Lambda \vert = r\}, \end{eqnarray}
in which $\vert \Lambda\vert$ is the cardinality of $\Lambda$, the norm used in (\ref{01.optimal}) is the operator norm, $1\leq r<m$ is a natural number and $\mathcal{D}_{r}$ is the set of all $m \times m$ diagonal matrices
with $r$ $1 ' s$ and $m-r$ $0 ' s$. Then, $ d_{r}(F,G) $ is the largest possible error
when $r$-erasures fall out. Indeed, $G$ is called an optimal dual frame of $F$
for $1$-erasure or $1$-loss optimal dual  if \begin{eqnarray}\label{1-erasure def}
 d_{1}(F,G) = \min \left\{  d_{1}(F,Y) : Y \textit{ is a dual of  }         F\right\}. \end{eqnarray}
Inductively, for $r>1$, a dual frame $G$ is called an \textit{optimal dual} of $F$  for
$r$-erasures ($r$-loss optimal dual) if it is optimal for $(r-1)$-erasures and
\begin{eqnarray*}
 d_{r}(F,G) = \min \left\{  d_{r}(F,Y) : Y \textit{ is a dual of  }         F\right\}. \end{eqnarray*} See \cite{Aa2018, holms, leng, lopez, Saliha} and references therein for more details and information on optimal reconstruction problem and  identification of optimal dual frames.

This work, was motivated by some recent methods of perfect recovery  of signals from   erasures corrupted frame coefficients at known or unknown locations \cite{han2020, han2}. In all previous approaches presented  on classical frames, erasures considered in frame coefficients (encoding frame coefficients). 
We are going to extend perfect recovery problem on $K$-frame theory,  however we show that    the methods previously used
 does not meet the requirements of reconstruction  in    $K$-frame setting.  Hence, we consider the  erasure coefficients on $K$-dual coefficients (as encoding frame instead of $K$-frame).  Then  we  introduce a new concept, called $(r,k)$-matrix, to recover lost data and
  get perfect reconstruction.   This also leads to some new method for recovery  problem in ordinary frames,  that
 sometimes  work for frames better than the previous methods. Among other things, 
 we demonstrate the adventages of using the new method in existence and construction with respect to the previous approaches.

The present  paper is organized as follows. In Section 2, we recall some definitions and notations of finite $K$-frames. In Section 3,
we provide more characterizations and structures of $K$-duals and particularly, the  canonical $K$-duals. We  present some concepts such as    minimal redundancy condition and maximal robustness for $K$-frames and provide some necessary conditions for  a finite set of indices   which satisfies    minimal redundancy condition. Moreover, we  discuss  the  robustness   of $K$-frame under operator perturbation, particularly when the erasure set satisfies  the  minimal redundancy condition or the  $K$-frame is maximal robust, in Section 4. Then in Section 5, we introduce a new matrix called $(r,k)$-matrix and give the necessary condition for the existence of $(r,k)$-matrices. This notion leads to a new matrix equation which  allows the signal vectors underlying the range space of a bounded operator to be exactly recovered. This approach not only assures that $K$-frames with uniform excess  under some erasures of $K$-dual coefficients make  complete reconstruction, but also provides a new method for erasure recovery  by using  ordinary frames,  by changing encode and decode frames, which
 sometimes  work better than the previous methods.    Finally, in Section 6,
we  exhibit  several examples to illustrate our results and the advantage of using  $(r,k)$-matrices.

\section{Finite $K$-frames}
 Atomic decomposition for a closed subspace $\mathcal{H}_{0}$ of a Hilbert space $\mathcal{H}$, as a new approach for reconstruction, was introduced by  Feichtinger et al.  
 with frame-like properties \cite{Han}. However, the sequences in  atomic decompositions
do not necessarily  belong to $\mathcal{H}_{0}$,  this striking property is valuable, especially   in sampling theory   \cite{Paw, Werthrt}. Then, $K$-frames were introduced  to study atomic systems with respect to a bounded operator $K\in B(\mathcal{H})$ \cite{Gav07}. In fact, $K$-frames   are equivalent with atomic systems for the operator $K$  and  help us to reconstruct elements from the range of a bounded linear operator $K$ in a separable Hilbert space.
In the sequel, we recall  some definitions and notations of finite $K$-frames.  A sequence $F:=\lbrace f_{i}\rbrace_{i \in I_{m}} \subseteq \mathcal{H}_{n}$ is called a $K$-frame for $\mathcal{H}_{n}$, if $R(K)\subset R(\theta_{F})$  or equivalently there exist constants $A, B > 0$ such that
\begin{eqnarray}\label{001}
A \Vert K^{*}f\Vert^{2} \leq \sum_{i\in I_{m}} \vert \langle f,f_{i}\rangle\vert^{2} \leq B \Vert f\Vert^{2}, \quad (f\in \mathcal{H}_{n}).
\end{eqnarray}
If $K$ is an onto operator then $F$ is an ordinary frame and therefore, $K$-frames arise  as a generalization of the ordinary frames. The constants $A$ and $B$ in $(\ref{001})$ are called the lower and the upper bounds of $F$, respectively. Similar to ordinary frames, the synthesis operator can be defined as $\theta_{F}: l^{2}(I_{m})\rightarrow \mathcal{H}$; $\theta_{F}(\{ c_{i}\}_{i\in I_{m}}) = \sum_{i\in I_{m}} c_{i}f_{i}$. A matrix representation for this bounded operator is the  matrix  $F_{n \times m}$ whose $i$th column is the $i$th $K$-frame vector, i.e.,
\begin{eqnarray*}
F= [f_{1}, ... f_{m}].
\end{eqnarray*}
Notice that, we sometimes denote a $K$-frame $F=\lbrace f_{i}\rbrace_{i \in I_{m}}$ by its synthesis matrix $F$.
Also, the analysis operator is given by $\theta_{F}^{*}(f)= \{ \langle f,f_{i}\rangle\}_{i\in I_{m}}$ and has the matrix representation as $F^{*}$.
The frame operator is given by  $S_{F} = \theta_{F}\theta_{F}^{*}$ with the matrix representation as $FF^{*}$ and $\mathcal{G}_{F} = F^{*} F$ denotes the Gramian matrix with respect to the $K$-frame $F$. Unlike ordinary frames, the frame operator of a K-frame is not invertible in general. Although,   in finite dimensional Hilbert spaces, $K$ is a closed range operator so $S_{F}$ from $R(K)$ onto $S_{F}(R(K))$ is an invertible operator \cite{Xiao}. When we need   this restriction of the $K$-frame operator we use the notation $S_{F}\vert R(K)$.
Suppose $M_{K}$ denotes matrix representation of the operator $K\in B(\mathcal{H}_{n})$ with respect to the standard orthonormal basis of $\mathcal{H}_{n}$. Then, a $K$-frame is said to be $\alpha$-tight whenever $FF^{*}=\alpha M_{K}M_{K}^{*}$, Parseval if $\alpha=1$ and equal norm (EN) if the columns of $F$ have the equal norm.

The authors in \cite{arefi3} considered the notion of duality for $K$-frames and presented several methods for construction and characterization of $K$-frames and their duals. Indeed, a Bessel sequence
 $\{g_{i}\}_{i \in I_{m}}\subseteq \mathcal{H}_{n}$ is called a \textit{$K$-dual} of $\{ f_{i} \}_{i\in I_{m}}$ if
\begin{eqnarray}\label{dual1}
Kf = \sum_{i\in I_{m}} \langle f,g_{i}\rangle f_{i}, \quad (f\in \mathcal{H}_{n}),
\end{eqnarray}
or equivalently $G$ is $K$-dual of $F$ if  $FG^{*}=  M_{K}$.
 The following result  is useful for the proof of  our main results.


\begin{thm}[Douglas   \cite{Douglas}]\label{equ0}
Let $L_{1}\in B(\mathcal{H}_{1}, \mathcal{H})$ and $L_{2}\in B(\mathcal{H}_{2}, \mathcal{H})$ be bounded linear mappings on given Hilbert spaces. Then the following assertions are equivalent:
\begin{itemize}
\item[(i)]
$R(L_{1}) \subseteq R(L_{2})$;
\item[(ii)]
$L_{1}L_{1}^{*} \leq \lambda^{2}L_{2}L_{2}^{*}$, \quad for some $\lambda > 0$;
\item[(iii)]
There exists a bounded linear mapping $X\in L(\mathcal{H}_{1}, \mathcal{H}_{2})$, such that $L_{1} = L_{2}X$.
\end{itemize}
Moreover, if (i), (ii) or (iii) are valid, then there exists a unique operator $X$ so that
\begin{itemize}
\item[(a)]
$\Vert X\Vert ^{2} = \inf \{\alpha>0, L_{1}L_{1}^{*}\leq \alpha L_{2}L_{2}^{*}\}$;
\item[(b)]
$N( L_{1}) = N(X)$;
\item[(c)]
$R(X) \subset \overline{R( L_{2}^{*})} $.
\end{itemize}
\end{thm}
For every $K$-frame $F=\{f_{i}\}_{i\in I_{m}}$ of $\mathcal{H}_{n}$
using the Douglas' theorem, there exists a unique operator $X_{F}\in B(\mathcal{H}_{n}, \mathbb{C}^{m})$ so that $\theta_{F}X_{F}=K$ and
\begin{eqnarray}\label{XW}
\Vert X_{F}\Vert ^{2} = \inf \{\alpha>0, \Vert K^{*}f\Vert ^{2}\leq \alpha\Vert \theta_{F}^{*}f\Vert ^{2}; f\in \mathcal{H}_{n}\}.
\end{eqnarray}
Moreover, $\{X_{F}^{*}\delta_{i}\}_{i\in I_{m}}$ is a $K$-dual of $F$ which its analysis operator obtains the minimal norm and is called the canonical $K$-dual. See \cite{Guo}.
For further information in $K$-frame theory  we refer the reader  to  \cite{arefi3, Han, Gav07, Gav08,  Xiao}.

Throughout this paper, we suppose that $\mathcal{H}_{n}$ is an $n$-dimensional  Hilbert space, $I_{m}=\{1, 2, ..., m\}$ and     $\{\delta_{i}\}_{i\in I_{m}}$ is the standard orthonormal basis of $l^{2}(I_{m})$.  For two Hilbert spaces $\mathcal{H}_{1}$ and $\mathcal{H}_{2}$, we denote by $B(\mathcal{H}_{1},\mathcal{H}_{2})$ the collection of all bounded linear operators between $\mathcal{H}_{1}$ and $\mathcal{H}_{2}$, and we abbreviate $B(\mathcal{H},\mathcal{H})$ by $B(\mathcal{H})$. Matrix representation associated with an operator $T$ is denoted by $M_{T}$ and the operator associated with a matrix $M$ is denoted by $T_{M}$.
 Also, we denote the range of $K\in B(\mathcal{H}_{n})$ by $R(K)$ and pseudo inverse of $K$ by $K^{\dagger}$. For a  subspace $V \subseteq \mathcal{H}_{n}$  the identity operator on $V$  and the orthogonal projection of $\mathcal{H}$ onto $V$  are denoted by   $I_{V}$ and  $\pi_{V}$, respectively.


 \section{Identification of the canonical $K$-dual}
In this section, we are going to obtain more structures of $K$-duals and particularly the canonical $K$-dual of   $K$-frames. For convenience, we denote the set of all  $K$-dual frames of $F=\{f_{i}\}_{i\in I_{m}}$ by $KD_{F}$. Obviously, $KD_{F}$ is a closed convex subset of $\mathcal{H}_{n}^{m}$, the set of all $m$-tuples of vectors in $\mathcal{H}_{n}$.
In the following, we    obtain the canonical $K$-dual in a new form which is more useful in the proof of our results.

\begin{lem}\label{1}
Let $F=\{f_{i}\}_{i\in I_{m}}$ be a $K$-frame of $\mathcal{H}_{n}$. With the above notations, there exists a unique bounded operator $\Gamma_{F}\in B(\mathcal{H}_{n})$ so that $\{\Gamma_{F}^{*}f_{i}\}_{i\in I_{m}}=\{X_{F}^{*}\delta_{i}\}_{i\in I_{m}}$.
\end{lem}
\begin{proof}
Using   Douglas' theorem,  there is  a unique operator $X_{F}\in B(\mathcal{H}_{n},l^{2}(I_{m}))$ so that $\theta_{F}X_{F}=K$  and $R(X_{F})\subseteq \overline{R(\theta_{F}^{*})}=R(\theta_{F}^{*})$. So by reusing   Douglas' theorem there exists a unique bounded operator $\Gamma_{F}\in B(\mathcal{H}_{n})$ so that $X_{F}=\theta_{F}^{*}\Gamma_{F}$ and
\begin{eqnarray*}
\Vert \Gamma_{F}\Vert ^{2} = \inf \{\alpha>0, \Vert X_{F}^{*}f\Vert ^{2}\leq \alpha\Vert \theta_{F}f\Vert ^{2}; f\in \mathcal{H}_{n}\}.
\end{eqnarray*}
Moreover,  we have 
\begin{eqnarray*}
\Gamma_{F}^{*}f_{i}=
\Gamma_{F}^{*}\theta_{F}\delta_{i}=X_{F}^{*}\delta_{i},
\end{eqnarray*}
for all $i\in I_{m}$. Hence, $\{\Gamma_{F}^{*}f_{i}\}_{i\in I_{m}}$ is exactly  the canonical $K$-dual of $F$.
\end{proof}
Easily, it can be checked that a sequence $G=\lbrace g_{i}\rbrace_{i \in I_{m}}$ is a $K$-dual of $F$ if and only if   $g_{i} = \Gamma_{F}^{*}f_{i}+u_{i}$, for all $i\in I_{m}$ where $U=\{u_i\}_{i\in I_{m}}$ satisfies $\theta_{F} \theta^{*}_{U}=0$.
\begin{lem}\label{2}
Let $F=\{f_{i}\}_{i\in I_{m}}$ be a $K$-frame of $\mathcal{H}_{n}$ and $G=\{g_{i}\}_{i\in I_{m}}$ be a $K$-dual of $F$. Then $G$ is the canonical $K$-dual  if and only if $S_{G}=\theta_{G} \theta^{*}_{Z}$ for every $K$-dual $Z$ of $F$.
\end{lem}
\begin{proof}
 Suppose $G$ is the canonical $K$-dual  and $Z$ is a $K$-dual of $F$, so by lemma \ref{1}  there exists a unique bounded operator $\Gamma_{F}\in B(\mathcal{H}_{n})$, so that $G=\{\Gamma_{F}^{*}f_{i}\}_{i\in I_{m}}$ so,
\begin{eqnarray*}
\theta_{G}(\theta_{G}^{*}- \theta_{Z}^{*})=\Gamma_{F}^{*}\theta_{F}(\theta_{G}^{*}- \theta
_{Z}^{*})=0.
\end{eqnarray*}
Thus,  $S_{G}=\theta_{G} \theta^{*}_{Z}$ for every $K$-dual $Z$ of $F$. Conversely, if  for every $K$-dual $Z$ of $F$ we have $S_{G}=\theta_{G} \theta^{*}_{Z}$. Then
\begin{eqnarray*}
\Vert \theta_{G}^{*}\Vert^{2}=\Vert \theta_{G}\theta_{G}^{*}\Vert
=\Vert \theta_{G} \theta_{Z}^{*}\Vert\leq \Vert  \theta^{*}_{G}\Vert \Vert  \theta_{Z}^{*}\Vert.
\end{eqnarray*}
This immediately implies that $\Vert \theta_{G}^{*}\Vert \leq \Vert  \theta_{Z}^{*}\Vert$, i.e., the analysis operator of $G$ has minimal norm and the  proof is complete.
\end{proof}
In the case that $F$ is a Parseval $K$-frame Lemma \ref{2} can be  reduced to a result of \cite{Xiang}. This result helps us to obtain the canonical $K$-dual of some classes of $K$-frames.
\begin{thm}\label{1thm}
Let $F=\{f_{i}\}_{i\in I_{m}}$ be a $K$-frame of $\mathcal{H}_{n}$. Then the followings hold;
\begin{itemize}
\item[(i)]
If $F\subseteq R(K)$  then $\{K^{*}(S_{F}\vert_{R(K)})^{-1}\pi_{S_{F}(R(K))}f_{i}\}_{i\in I_{m}}$ is the canonical $K$-dual of $F$.
\item[(ii)] Either $R(K)\subseteq S_{F}(R(K))$ or $F\subseteq S_{F}(R(K))$ implies that the Bessel  sequence $\{K^{*}((S_{F}\vert_{R(K)})^{-1})^{*}\pi_{R(K)}f_{i}\}_{i\in I_{m}}$ is the canonical $K$-dual of $F$.
\end{itemize}
\end{thm}
\begin{proof}
To show $(i)$, first  we note that by the assumption $F\subseteq R(K)$ is a $K$-frame of $\mathcal{H}_{n}$ and so   $G:=\{K^{*}(S_{F}\vert_{R(K)})^{-1}\pi_{S_{F}(R(K))}f_{i}\}_{i\in I_{m}}$ is a $K$-dual of $\pi_{R(K)}F=F$ by using  Proposition 2.3 of \cite{arefi3}. Moreover,
\begin{eqnarray*}
S_{G}&=&  K^{*}(S_{F}\vert_{R(K)})^{-1}\pi_{S_{F}(R(K))}S_{F}((S_{F}\vert_{R(K)})^{-1})^{*}K\\
 &=&  K^{*}(S_{F}\vert_{R(K)})^{-1}\pi_{S_{F}(R(K))}\pi_{R(K)}S_{F}((S_{F}\vert_{R(K)})^{-1})^{*}K\\
&=& K^{*}(S_{F}\vert_{R(K)})^{-1}\pi_{S_{F}(R(K))}(S_{F}\vert_{R(K)})^{*}((S_{F}\vert_{R(K)})^{-1})^{*}K\\
&=& K^{*}(S_{F}\vert_{R(K)})^{-1}\pi_{S_{F}(R(K))}K\\
&=& K^{*}(S_{F}\vert_{R(K)})^{-1}\pi_{S_{F}(R(K))}\theta_{F} \theta_{Z}^{*}\\
&=&\theta_{G}  \theta_{Z}^{*},
\end{eqnarray*}
for every $K$-dual $Z$ of $F$.
Now, assume that $R(K)\subseteq S_{F}(R(K))$. Since the operator $S_{F}\vert_{R(K)} :R(K) \rightarrow S_{F}(R(K))$ is  invertible so $R(K)= S_{F}(R(K))$. Thus
\begin{eqnarray*}
Kf&=&  S_{F}\vert_{R(K)}  (S_{F}\vert_{R(K)})^{-1}Kf\\
 &=&  \sum_{i\in I_{m}}\langle (S_{F}\vert_{R(K)})^{-1}Kf,f_{i}\rangle f_{i}\\
&=& \sum_{i\in I_{m}}\langle f,K^{*}((S_{F}\vert_{R(K)})^{-1})^{*}\pi_{(R(K))}f_{i}\rangle f_{i},
\end{eqnarray*}
for every $f\in \mathcal{H}_{n}$. Hence, $G:=\{K^{*}((S_{F}\vert_{R(K)})^{-1})^{*}\pi_{R(K)}f_{i}\}_{i\in I_{m}}$ is a $K$-dual of $F$.  Moreover,
\begin{eqnarray*}
S_{G} &=& K^{*}((S_{F}\vert_{R(K)})^{-1})^{*}K\\
&=&K^{*}((S_{F}\vert_{R(K)})^{-1})^{*}\pi_{R(K)}K\\
&=&K^{*}((S_{F}\vert_{R(K)})^{-1})^{*}\pi_{R(K)}\theta_{F}  \theta_{Z}^{*}\\
&=& \theta_{G}  \theta_{Z}^{*},
\end{eqnarray*}
for every $K$-dual $Z$ of $F$. The above computations along with Lemma \ref{2}, follows the desired result.
Finally, if $F\subseteq S_{F}(R(K))$ is a $K$-frame. Then, we have $R(K)\subseteq  \textit{span}\{f_{i}\}_{i\in I_{m}} \subseteq S_{F}(R(K))$ that is similar to the previous case. Thus,  the proof is complete.
\end{proof}
The converse of Theorem \ref{1thm}, does not hold in general. To see this   and also
 the importance of the sufficiency conditions in Theorem \ref{1thm},   see  Example \ref{1ex1} and Example \ref{ex1} in Section 6.

\begin{rem}\label{3.4./}
The structure of the canonical $K$-dual of a Parseval $K$-frame $F$ is $K^{\dagger}F$. See \cite{Miao}. Indeed, in this case $\Gamma_{F}=(K^{\dagger})^{*}$. Also, in this regard, for a $K$-frame  $F\subseteq R(K)$ we have that $\Gamma_{F}=((S_{F}\vert_{R(K)})^{-1})^{*}K$.\end{rem}


 \section{Minimal redundancy condition}
In this section, we provide the concept of   minimal redundancy condition  and maximal robust   for $K$-frames
 and give some necessary conditions for  a finite set of indices  which satisfies  minimal redundancy condition.  Then, we discuss   the problem of  robustness  under operator perturbation of $K$-frame, particularly when the erasure set satisfies  the  minimal redundancy condition or the  $K$-frame is maximal robust. For more information of these  concepts  on  classical frames we refer the reader to \cite{Alexeev, Arambasic, Kovaevi, larson}

Suppose $F$ denotes the associated matrix of a $K$-frame $\{f_{i}\}_{i\in I_{m}}$ in Hilbert space $\mathcal{H}_{n}$.  A finite set of indices $\sigma\subset I_{m}$ satisfies the minimal redundancy condition (MRC) for $F$ whenever $\{f_{i}\}_{i\in \sigma^{c}}$ is  a $K$-frame for $\mathcal{H}_{n}$. Furthermore,  we say $F$ satisfies MRC for $r$-erasures if every subset $\sigma\subset I_{m}$, $\vert \sigma\vert=r$ satisfies MRC for $F$. Also,  $F$ is said to be of uniform excess $r$ if  it is an exact $K$-frame when $r$ columns of $F$ are removed
 and $F$ is called maximal robust (MR) if every $r_{k}$ columns of $F$ is an exact $K$-frame, where $r_{k}:=rank K$.
Note that, for a $K$-frame that is MR, every submatrix $n \times r_{k}$ has a left inverse. However, the converse does not hold, in general. For instance, in Example \ref{ex1}, $rank K=2$ and every $2$ columns of $F$ are linearly independent so every submatrix of $F$ containing $2$ columns has a left inverse,  but $\{f_{3},f_{4}\}$ is not a $K$-frame.
In what follows, we give some necessary conditions for  a finite set of indices $\sigma\subset I_{m}$ which satisfies MRC. To be convenient,  we use  $\theta_{\sigma}$, $S_{\sigma}$ and $\xi^{*}_{\sigma}$ to denote the synthesis operator,
frame operator of a $K$-frame and  the analysis operator of the canonical $K$-dual  whenever the index set is limited to $\sigma$.

 \begin{thm}\label{fthm}
Suppose $F=\{f_{i}\}_{i\in I_{m}}$ is a $K$-frame of $\mathcal{H}_{n}$ and $\sigma\subset I_{m}$ satisfies MRC for $F$. Then
\begin{itemize}
\item[(i)]
$R(\theta_{F}^{*}K)\cap \textit{span}\{\delta_{i}\}_{i\in \sigma}=\{0\}$.
\item[(ii)] If $F$ is a Parseval $K$-frame then $\left(K-\theta_{\sigma}\xi_{\sigma}^{*} \right)\vert_{R(K^{\dagger})}$ is an invertible operator from $R(K^{\dagger})$ onto $S_{\sigma^{c}}(R(K))$.
\end{itemize}
\end{thm}
\begin{proof}
To show $(i)$, on the contrary, assume that  there exists a non-zero element $\alpha \in R(\theta_{F}^{*}K)\cap \textit{span}\{\delta_{i}\}_{i\in \sigma}$. Then, there exists $f\in \mathcal{H}_{n}$, $\{c_{i}\}_{i\in \sigma}\subseteq \mathbb{C}$ so that
\begin{eqnarray*}
\alpha=\theta_{F}^{*}Kf=\sum_{i \in \sigma}c_{i}\delta_{i}.
\end{eqnarray*}
Thus, $\theta_{F}^{*}Kf \perp \delta_{i}$ for every $i\in \sigma^{c}$ and so
\begin{eqnarray*}
\langle Kf,f_{i}\rangle = \langle Kf,\theta_{F}\delta_{i}\rangle = \langle \theta_{F}^{*} Kf,\delta_{i}\rangle =0,
\end{eqnarray*}
for every $i\in \sigma^{c}$. Hence  $Kf \perp \{f_{i}\}_{i \in \sigma^{c}}$ that is a contradiction. This follows the desired result.

Now, let $F$ be a  Parseval $K$-frame then
\begin{eqnarray*}
K-\theta_{\sigma}\xi_{\sigma}^{*} =\theta_{\sigma^{c}}\xi_{\sigma^{c}}^{*}=S_{\sigma^{c}}  (K^{\dagger})^{*},
\end{eqnarray*}
where the last equality is obtained by Remark \ref{3.4./} so it is sufficient to prove that $S_{\sigma^{c}}  (K^{\dagger})^{*}\vert_{R(K^{\dagger})}$ is an invertible operator. Since $\sigma$ satisfies MRC the operator $S_{\sigma^{c}}\vert_{R(K)}$ is  invertible from $R(K)$ onto $S_{\sigma^{c}}(R(K))$. Consider $\Gamma_{\sigma^{c}}:=K^{*}(S_{\sigma^{c}}\vert_{R(K)})^{-1}$, we show that $\Gamma_{\sigma^{c}}$ is the inverse of the operator $S_{\sigma^{c}}  (K^{\dagger})^{*}\vert_{R(K^{\dagger})}$. Indeed
\begin{eqnarray*}
&&\Gamma_{\sigma^{c}}S_{\sigma^{c}} (K^{\dagger})^{*}f\\
 &=& K^{*}(S_{\sigma^{c}}\vert_{R(K)})^{-1}S_{\sigma^{c}}\vert_{R(K)} (K^{\dagger})^{*}f\\
&=& K^{*}(K^{\dagger})^{*}f \\
&=& (K^{\dagger}K)^{*}f\\
&=& K^{\dagger}Kf =f,
\end{eqnarray*}
for every $f\in R(K^{\dagger})$. Thus $\Gamma_{\sigma^{c}}S_{\sigma^{c}} (K^{\dagger})^{*}\vert_{R(K^{\dagger})}=I_{{R(K^{\dagger})}}$. On the other hand,
 \begin{eqnarray*}
&&S_{\sigma^{c}} (K^{\dagger})^{*}\Gamma_{\sigma^{c}}f\\
 &=& S_{\sigma^{c}} (K^{\dagger})^{*}K^{*}(S_{\sigma^{c}}\vert_{R(K)})^{-1}f\\
&=&  S_{\sigma^{c}} K K^{\dagger}(S_{\sigma^{c}}\vert_{R(K)})^{-1}f\\
&=& S_{\sigma^{c}}\vert_{R(K)}(S_{\sigma^{c}}\vert_{R(K)})^{-1}f=f,
\end{eqnarray*}
for every $f\in S_{\sigma^{c}}(R(K))$. Hence, $S_{\sigma^{c}} (K^{\dagger})^{*}\Gamma_{\sigma^{c}}\vert_{S_{\sigma^{c}}(R(K))}=I_{S_{\sigma^{c}}(R(K))}$.
This implies the desired result.
\end{proof}
It is worth noting that  the condition $(i)$ in the above theorem is not  sufficient for a subset $\sigma$ to satisfy MRC.  See Example \ref{4.1ex};
moreover,
by applying Theorem \ref{fthm}, if $\sigma\subset I_{m}$ satisfies MRC, we get some $K$-frames and $K^{\dagger}$-frame with $K^{\dagger}$-dual on the remained index set $\sigma^{c}$.
\begin{cor}
Let $F=\{f_{i}\}_{i\in I_{m}}$ be a   $K$-frame of $\mathcal{H}_{n}$ and $\sigma\subset I_{m}$ satisfies MRC then
\begin{itemize}
\item[(i)]
$\{ K^{*}(S_{\sigma^{c}}\vert_{R(K)})^{-1}\pi_{S_{\sigma^{c}}(R(K))}f_{i}\}_{i\in \sigma^{c}}$ is a $K^{\dagger}$-frame with $K^{\dagger}$-dual $\{(K^{\dagger})^{*}K^{\dagger}f_{i}\}_{i\in \sigma^{c}}$.
\item[(ii)]$\{(K^{\dagger})^{*}K^{\dagger}f_{i}\}_{i\in \sigma^{c}}$ is also a $K$-frame for $\mathcal{H}_{n}$.
\end{itemize}
\end{cor}
\begin{proof}
Since $S_{\sigma^{c}}\vert_{R(K)}$ is invertible we have that
\begin{eqnarray*}
K^{\dagger}f&=& K^{*}(S_{\sigma^{c}}\vert_{R(K)})^{-1}S_{\sigma^{c}}\vert_{R(K)}(K^{\dagger})^{*}K^{\dagger}f\\
&=& K^{*}(S_{\sigma^{c}}\vert_{R(K)})^{-1}\pi_{S_{\sigma^{c}}(R(K))}\sum_{i\in \sigma^{c}}\langle (K^{\dagger})^{*}K^{\dagger}f, f_{i}\rangle f_{i}\\
&=&\sum_{i\in \sigma^{c}}\langle (K^{\dagger})^{*}K^{\dagger}f, f_{i}\rangle K^{*}(S_{\sigma^{c}}\vert_{R(K)})^{-1}\pi_{S_{\sigma^{c}}(R(K))}f_{i},
\end{eqnarray*}
for every $f\in \mathcal{H}_{n}$. Hence, $(i)$ is obtained by Lemma 2.2 of \cite{arefi3}. Using the above computations and the fact that
\begin{eqnarray*}
R(K)= R(K^{\dagger})^{*} \subseteq \textit{span}\{(K^{\dagger})^{*}K^{\dagger}f_{i}\}_{i\in \sigma^{c}},
\end{eqnarray*}
we get $(ii)$.
\end{proof}

\begin{thm}
Let $F$ be the associated matrix of a  $K$-frame for $\mathcal{H}_{n}$. Then the following assertions  hold, where in all matrix products below, we let the sizes  be compatible.
\begin{itemize}
\item[(i)]
$AFU$ is $T_{A}K$-frame for any square  matrix $A$ and a unitary matrix $U$. In particular $FU$ is a $K$-frame and $GU\in KD_{AFU}$ for every $G\in KD_{F}$.
\item[(ii)] If $A$ is invertible and $U$ is a unitary matrix then  $G\in KD_{F}$ if and only if $GU\in KD_{AFU}$.
\item[(iii)]If $F$ is $\alpha$-tight $K$-frame then $AFU$ is $T_{A}K$-frame for any  matrix $A$ and a unitary matrix $U$. Moreover, $FU$ is an $\alpha$-tight $K$-frame.
\item[(iv)]If $F$ is EN then $UFD$ is also EN for any  unitary matrix $U$ and   unitary diagonal matrix $D$.
\item[(v)] If $F$ is MR $K$-frame then $AFD$ as an $T_{A}K$-frame is MR  for any  invertible matrix $A$ and   unitary diagonal matrix $D$.
\item[(vi)] If $F$ satisfies MRC for $r$-erasures then $AFD$ as a $T_{A}K$-frame satisfies MRC for $r$-erasures  for any unitary diagonal matrix  $D$  and  square matrix $A$.
\end{itemize}
\end{thm}
\begin{proof}
Suppose $\gamma$ is a lower $K$-frame  bound of $F$. Then by the assumption in $(i)$ we obtain
\begin{eqnarray*}
AFU(AFU)^{*}&=& AFUU^{*}F^{*}A^{*}\\
&=&AFF^{*}A^{*}\\
&\geq& \gamma AM_{K}M_{K}^{*}A^{*}\\
&=& \gamma  A M_{K} (A M_{K})^{*}.
\end{eqnarray*}
The existence of the upper bound is clear, so $AFU$ is $T_{A}K$-frame of $\mathcal{H}_{n}$.   Moreover, if $A$ is the  identity matrix, $FU$ is a $K$-frame of $\mathcal{H}_{n}$. On the other hand,  for every $G\in KD_{F}$ we have that
\begin{eqnarray*}
AFUU^{*}G^{*}=AFG^{*}=AM_{K},
\end{eqnarray*}
 so $GU\in KD_{AFU}$  and $(i)$ is proved.   
The cases  $(ii)$, $(iii)$ and $(iv)$ are proved by definitions  and some straightforward computations. For $(v)$, we note that  $AFD$ is $T_{A}K$-frame by $(i)$, so we  only show that $AFD$   is MR. Indeed, let $F_{n \times m}$ be MR $K$-frame,  $A_{n \times n}$ and $D_{m \times m}$  invertible and diagonal unitary matrices, respectively. Then a submatrix $n\times r_{k}$ of $AFD$ is as $A\mathcal{M}\mathcal{N}$ where $\mathcal{M}_{n\times r_{k}}$ is a submatrix of $F$ and $\mathcal{N}_{r_{k}\times r_{k}}$ is a diagonal submatrix of $D$. Hence, the columns of  $\mathcal{M}$ constitute an exact $K$-frame and so $\mathcal{M}$ has a left inverse. This implies that $A\mathcal{M}\mathcal{N}$  also has a left inverse i.e., its columns are linearly independent and generate $R(T_{A}K)$. Moreover, this vector columns constitute an exact $T_{A}K$-frame. Thus, $AFD$  is MR.

Finally, let $F$ satisfies MRC for $r$-erasures,  $A$ and $D$ be arbitrary $n \times n$ matrix and $m\times m$  unitary diagonal matrix, respectively. A submatrix $n\times (m-r)$ of $AFD$ is as $A\mathcal{M}\mathcal{N}$ where $\mathcal{M}_{n\times (m-r)}$ is a submatrix of $F$ and $\mathcal{N}_{(m-r)\times (m-r)}$ is a diagonal submatrix of $D$. Since $\mathcal{M}$ is a $K$-frame, applying the assumption,    one immediately obtains that $A\mathcal{M}\mathcal{N}$ is also a $T_{A}K$-frame  by $(i)$. This completes the proof.
\end{proof}
\section{Perfect reconstructions by $(r,k)$-matrices}
In what follows, we present some matrix methods  which lead to fewer errors   if  $K$-frame is of uniform excess   or even we have perfect  reconstruction  under erasures. To this end, we present two approaches, that the first one is motivated by \cite{han2020, han2}, however unlike ordinary frames, for $K$-frames it does not work very well. Hence, we set a new concept so called $(r,k)$-matrix to get perfect reconstruction in this case. Also, we show this approach work for frames sometimes better than  the previous methods.

Let  $F=\{f_{i}\}_{i\in I_{m}}$ be a $K$-frame of $\mathcal{H}_{n}$ with uniform excess $r$ and $G=\{g_{i}\}_{i\in I_{m}}$ be $K$-dual  of $F$. Since $\{f_{i}\}_{i=r+1}^{m}$ is an exact $K$-frame then for any $g_{i}$, $(1\leq i\leq r)$ there exist unique coefficients $\{\alpha_{i,j}\}_{j=r+1}^{m}\subset \mathbb{C}$ so that
\begin{eqnarray*}
\pi_{R(K)}g_{i} = \sum_{j=r+1}^{m} \alpha_{i,j}f_{j}, \quad (1\leq i \leq r).
\end{eqnarray*}
Consider
\begin{equation*}
M_{F} = \left[
 \begin{array}{ccc}

&1& \quad 0\quad  . \quad . \quad . \quad 0  \quad -\alpha^{*}_{1, r+1}\quad . \quad . \quad . \quad -\alpha^{*}_{1, m}\\

&0&\quad  1 \quad . \quad . \quad . \quad 0 \quad -\alpha^{*}_{2, r+1}\quad . \quad . \quad . \quad -\alpha^{*}_{2, m}\\

&&\quad   \quad  \quad  \quad  \quad  \quad  \quad  \quad  \quad   \quad \\

&& \quad  \quad  \quad  \quad  \quad  \quad  \quad  \quad  \quad  \quad \\

&& \quad \quad \quad  \quad  \quad  \quad  \quad  \quad  \quad  \qquad \\

&0& \quad 0  \quad . \quad .  \quad . \quad 1 \quad -\alpha^{*}_{r, r+1}\quad . \quad . \quad . \quad -\alpha^{*}_{r, m}\\

\end{array} \right].
\end{equation*}
Then
\begin{equation*}
M_{F} \left[
 \begin{array}{ccc}

\langle f,\pi_{R(K)}g_{1}\rangle\\
.\\
.\\
.\\
\langle f,\pi_{R(K)}g_{r}\rangle \\
\\
\langle f,f_{r+1}\rangle\\
.\\
.\\
.\\
\langle f,f_{m}\rangle \\
\end{array} \right] = 0,
\end{equation*}
for every $f\in \mathcal{H}_{n}$ and consequently
\begin{equation}\label{12}
M_{1} \left[
 \begin{array}{ccc}

\langle f,\pi_{R(K)}g_{1}\rangle\\
.\\
.\\
.\\
\langle f,\pi_{R(K)}g_{r}\rangle \\

\end{array} \right]+M_{2}\left[
 \begin{array}{ccc}
\langle f,f_{r+1}\rangle\\
.\\
.\\
.\\
\langle f,f_{m}\rangle \\
\end{array} \right] = 0,
\end{equation}
where $M_{1}$  is the submatrix  consisting of the first $r$ columns of $M_{F}$,   and $M_{2}$ is  the submatrix consisting of the rest columns. This assures that for any $r$-erasures of $K$-dual coefficients  $\{\langle f, g_{i}\rangle\}_{i\in \Lambda}$, $\vert \Lambda\vert=r$ we may  recover the coefficients $\{\langle f, \pi_{R(K)}g_{i}\rangle\}_{i\in \Lambda}$ by solving the equation  $(\ref{12})$ as follows
\begin{equation}\label{13}
\left[
 \begin{array}{ccc}

\langle f,\pi_{R(K)}g_{1}\rangle\\
.\\
.\\
.\\
\langle f,\pi_{R(K)}g_{r}\rangle \\

\end{array} \right]= -M_{2}\left[
 \begin{array}{ccc}
\langle f,f_{r+1}\rangle\\
.\\
.\\
.\\
\langle f,f_{m}\rangle \\
\end{array} \right].
\end{equation}
Replacing the coefficients $\{\langle f, \pi_{R(K)}g_{i}\rangle\}_{i\in \Lambda}$ by $\{\langle f, \sum_{j\in \Lambda^{c}}\alpha_{i,j}f_{i}\rangle\}_{i\in \Lambda}$
and  using   the fact that the error operator is obtained by
\begin{eqnarray*}
E_{\Lambda}= \sum_{i\in \Lambda}f_{i} \otimes g_{i}
= \sum_{i\in \Lambda}f_{i} \otimes \pi_{R(K)}g_{i}+\sum_{i\in \Lambda}f_{i} \otimes \pi_{R(K)^{\perp}} g_{i},
\end{eqnarray*}
we get a reduced error operator as
\begin{eqnarray*}
\tilde{E}_{\Lambda}=E_{\Lambda}-\Delta_{\Lambda},
\end{eqnarray*}
where $\Delta_{\Lambda}=\sum_{i\in \Lambda}f_{i} \otimes \pi_{R(K)} g_{i}$. Equivalently, we have  $\tilde{E}_{\Lambda}f=\sum_{i\in \Lambda} \langle f, \pi_{R(K)^{\perp}} g_{i}\rangle f_{i}$, for every $f\in \mathcal{H}_{n}$.
Hence, for  computing of the error operator one needs only find a $K$-dual frame $G$ which satisfies
\begin{eqnarray}\label{new error}
\max_{\vert \Lambda\vert=r} \left\Vert\sum_{i\in \Lambda}f_{i} \otimes \pi_{R(K)^{\perp}} g_{i}\right \Vert=\min\left\{\max_{\vert \Lambda\vert=r}\left\Vert\sum_{i\in \Lambda}f_{i} \otimes \pi_{R(K)^{\perp}} h_{i}\right \Vert ;\quad \{ h_{i}\}_{i\in I_{m}}\in KD_{F}\right\}.
\end{eqnarray}
From this point of view, by  a $K$-frame with uniform excess property which has  a $K$-dual    $\{g_{i}\}_{i\in I_{m}}\subseteq R(K)$, we will have the perfect reconstruction. Otherwise, for every $K$-dual of $F$ which satisfies  $(\ref{new error})$ the error rate is reduced.

Now, we present a new method which  allows a perfect reconstruction.  Moreover, by this approach $K$-frames with uniform excess  under some erasures of $K$-dual coefficients make a  complete reconstruction   and this process is independent of the   choice of $K$-dual. We recall the spark of a matrix \cite{Alexeev} is the size of the smallest linearly dependent subset of the columns and the spark of a collection of vectors in a  finite dimensional Hilbert space is considered as the spark of its synthesis matrix. Moreover,  for any $m\times n$ matrix $A$
\begin{eqnarray}\label{sparkerim}
sparkA=\min\{\Vert x\Vert_{0} : Ax=0, x\neq 0\},
\end{eqnarray}
where $\Vert x\Vert_{0}$, the Humming weight of a vector $x=\{x_{i}\}_{i\in I_{n}}$, is defined as follows
\begin{eqnarray*}
\Vert x\Vert_{0}=\left\vert \{j\in I_{n} : x_{j}\neq 0\}\right \vert.
\end{eqnarray*}
 See \cite{Alexeev, Donoho} for more information.
Let $F=\{f_{i}\}_{i\in I_{m}}$ be a   $K$-frame of $\mathcal{H}_{n}$ with a $K$-dual $G=\{g_{i}\}_{i\in I_{m}}$. Then we have that
\begin{eqnarray*}
\sum_{i\in I_{m}}\langle f_{i},f_{j}\rangle  \langle f ,g_{i}\rangle  =\langle Kf,f_{j}\rangle,
\end{eqnarray*}
for all $j\in I_{m}$. Equivalently
\begin{equation*}
\left[
 \begin{array}{ccc}

\langle f_{1},f_{1}\rangle \quad   \langle f_{2},f_{1}\rangle \quad  . \quad .  \quad . \quad      \langle f_{m},f_{1}\rangle\\

\langle f_{1},f_{2}\rangle  \quad  \langle f_{2},f_{2}\rangle \quad  . \quad .  \quad . \quad   \langle f_{m},f_{2}\rangle\\

\quad  . \quad  \quad  \quad  \quad  \quad  \quad  \quad . \quad   \quad \\

\quad . \quad  \quad  \quad  \quad  \quad  \quad  \quad  .\quad  \quad \\

\quad \quad \quad  \quad  \quad  \quad  \quad  \quad  \quad  \qquad \\

\langle f_{1},f_{m}\rangle \quad  \langle f_{2},f_{m}\rangle \quad . \quad .  \quad . \quad    \langle f_{m},f_{m}\rangle\\

\end{array} \right] \left[
 \begin{array}{ccc}

\langle f, g_{1}\rangle\\
\\
.\\
.\\
\\
\langle f,g_{m}\rangle \\
\end{array} \right]= \left[
 \begin{array}{ccc}

\langle Kf,f_{1}\rangle\\
\\
.\\
.\\
\\
\langle Kf,f_{m}\rangle \\
\end{array} \right],
\end{equation*}
subsequently we get
\begin{eqnarray}\label{Matrix eq0}
\mathcal{G}_{F}G^{*} =F^{*}M_{K}.
\end{eqnarray}
This Motivates the following definition.
\begin{defn}\label{deforg}
Suppose that $F=\{f_{i}\}_{i\in I_{m}}$ is a   $K$-frame of $\mathcal{H}_{n}$ with a $K$-dual $G$. Then a $m \times m$ matrix $M_{F,G}$ with spark $r+1$ is called an $(r,k)$-matrix associated with   $F$ and $G$ whenever
\begin{equation}\label{r,kmatrix}
(M_{F,G}-\mathcal{G}_{F})G^{*}=0.
\end{equation}
\end{defn}
\begin{rem} Note that by Definition \ref{deforg},   every $K$-frame $F$ with non-zero vectors  has at least a $(1,k)$-matrix  $M_{F,G}=\mathcal{G}_{F}$ associated with   $F$ and an arbitrary $G\in KD_{F}$.
\end{rem}
The next result shows that for a  $K$-frame $F$, the existence of an $(r,k)$-matrix associated with $F$ and $G\in KD_{F}$ assures the unknown  $r$-erasures of $K$-dual coefficients can be completely recovered.
\begin{thm}\label{Matrix eq} Let $F=\{f_{i}\}_{i\in I_{m}}$ be a   $K$-frame of $\mathcal{H}_{n}$ with a  $K$-dual $G$ and $c=\{c_{i}\}_{i\in I_{m}}$ be a sequence of $K$-dual frame coefficients.
\begin{itemize}
\item[(i)] If there exists an $(r,k)$-matrix $M_{F,G}$ associated with   $F$ and $G$ then   any $r$-erasures of $K$-dual coefficients can be recovered by solving the equation
\begin{equation}\label{unkneq}
(M_{F,G}-\mathcal{G}_{F})c=0,
\end{equation}
\item[(ii)] If $sparkF=r+1$ then
 any $r$-erasures  of $K$-dual frame coefficients  can be recovered by solving the equation
$\mathcal{G}_{F}c =\theta_{F}^{*}K$.
\end{itemize}
\end{thm}
\begin{proof}
To show $(i)$, without losing the generality, we suppose  for an original vector $f$ the erasure   coefficients are $c_{1}:=\{c_{i}\}_{i=1}^{r}=\{\langle f,g_{i}\rangle \}_{i=1}^{r}$ and the remaining coefficients are $c_{2}:=\{\langle f,g_{i}\rangle \}_{i=r+1}^{m}$. Furthermore, let $M_{1}$ and $M_{2}$ denote submatrices of $M_{F,G}$ containing of the first $r$ columns and the rest, respectively. Then using $(\ref{unkneq})$ we get
\begin{equation*}
M_{1} c_{1}+M_{2}c_{2}=\mathcal{G}_{F}c= \theta_{F}^{*}\theta_{F}c.
\end{equation*}
Equivalently,
\begin{equation}\label{10002}
M_{1} \left[
 \begin{array}{ccc}

\langle f,g_{1}\rangle\\
.\\
.\\
.\\
\langle f,g_{r}\rangle \\

\end{array} \right]=\left[
 \begin{array}{ccc}

\langle Kf,f_{1}\rangle\\
.\\
.\\
.\\
\langle Kf,f_{m}\rangle \\
\end{array} \right]-M_{2}\left[
 \begin{array}{ccc}
\langle f,g_{r+1}\rangle\\
.\\
.\\
.\\
\langle f,g_{m}\rangle \\
\end{array} \right].
\end{equation}
Using the assumption that the columns of $M_{1}$ are linearly independent and so the pseudo inverse $ M_{1}^{\dagger}=(M_{1}^{T}M_{1})^{-1}M_{1}^{T}$ there exists \cite{Berut76}. Hence by $(\ref{10002})$ we obtain
\begin{equation*}
[c_{i}]_{i=1}^{r}=[\langle  f,g_{i}\rangle]_{i=1}^{r}= M_{1}^{\dagger}\left([\langle Kf,f_{i}\rangle]_{i=1}^{m}-M_{2}[\langle  f,g_{i}\rangle]_{i=r+1}^{m} \right).
\end{equation*}
Thus, the missing coefficients are obtained completely and we have the perfect reconstruction.
On the other hand, it is known that  
\begin{equation*}
Ker F=Ker F^{*}F=Ker \mathcal{G}_{F}.
\end{equation*}
 Therefore by $(\ref{sparkerim})$  we have that  $sparkF=spark\mathcal{G}_{F}$. Hence $M_{F,G}=\mathcal{G}_{F}$ is an $(r,k)$-matrix associated with   $F$ and $G$. Now, the proof of $(ii)$ is complete by using $(i)$. Note that this $(r,k)$-matrix is independent of $K$-dual $G$.
\end{proof}
\begin{cor}\label{excessd}
Let $F=\{f_{i}\}_{i\in I_{m}}$ be a   $K$-frame of $\mathcal{H}_{n}$ with uniform excess $r>0$ then every $m-r+1$ columns of $F$ be linearly dependent. Moreover,  any $(m-r)$-erasures  of $K$-dual coefficients   can be exactly recovered for every  $K$-dual of $F$.
\end{cor}
 \begin{proof}
 Since $F$ is with uniform excess $r$ so  any $m-r$ columns of associate matrix $F$ constitutes an exact $K$-frame for $\mathcal{H}_{n}$. Without loss of generality, let the first $m-r+1$ columns of $F$ are linearly independent. Moreover, assume that $R(K)=\textit{span}\{\eta_{i}\}_{i\in I_{l}}$. Then for every $\eta_{i}$ there exist unique coefficients $\{\alpha_{i,j}\}_{i=1}^{m-r}$ and $\{\alpha^{'}_{i,j}\}_{i=2}^{m-r+1}$ so that
\begin{eqnarray*}
\eta_{i}=\sum_{j=1}^{m-r}\alpha_{i,j}f_{j}=\sum_{j=2}^{m-r+1}\alpha^{'}_{i,j}f_{j}, \quad (i\in I_{l}) .
\end{eqnarray*}
By these equalities, and the assumption that $\{f_{i}\}_{i=1}^{m-r+1}$ is  linearly independent   we conclude that
\begin{eqnarray*}
&&\alpha_{i,1}=\alpha^{'}_{i,m-r+1}=0, \\
&&\alpha_{i,j}=\alpha^{'}_{i,j} \quad (2\leq j\leq m-r).
\end{eqnarray*}
Consequently,
\begin{eqnarray*}
\eta_{i}=\sum_{j=2}^{m-r}\alpha_{i,j}f_{j},
\end{eqnarray*}
for all $i\in I_{l}$. So $\{f_{i}\}_{i=2}^{m-r}$ is also  a $K$-frame  of $\mathcal{H}_{n}$
that is a contradiction. This follows that every $m-r+1$ columns of $F$ is  linearly dependent. The moreover part follows from Theorem \ref{Matrix eq} (ii) and the fact that $sparkF=m-r+1$.
\end{proof}
It is worth noticing that, Corollary \ref{excessd} for   $r=0$ fails. Indeed, if    $F$ is a   $K$-frame   with uniform excess $0$ then $F$  is an exact $K$-frame and so $sparkF=+\infty$.
Also, Theorem \ref{Matrix eq}   leads to a new approach for erasure recovery by using   ordinary frames in finite dimensional Hilbert spaces by changing encode and decode frames.
\begin{cor}\label{frameco}
Let $F=\{f_{i}\}_{i\in I_{m}}$ be a   frame of $\mathcal{H}_{n}$ with a   dual frame $G$ so that $sparkF=r+1$. Then
 any $r$-erasures  of  dual frame coefficients  can be recovered by solving the equation
$\mathcal{G}_{F}c =\theta_{F}^{*}$, where $c$ is a dual frame coefficient with
 unknowns  $\{c_{i}\}_{i\in \Lambda}$, $\vert \Lambda \vert =r$. \end{cor}
\begin{cor}
Let $F=\{f_{i}\}_{i\in I_{m}}$ be a   frame of $\mathcal{H}_{n}$ with a  dual frame $G$. Then
 any $1$-erasure  of  dual frame coefficients  can be recovered by
$\mathcal{G}_{F}c =\theta_{F}^{*}$,
for   unknown  erasure $c_{j}$, $j\in I_{m}$.
\end{cor}
The above corollaries illustrate the advantage and difference of using $(r,k)$-matrix and erasure recovery matrix \cite{han2020}. Indeed, if $(F,G)$ is a pair of dual frames for $\mathcal{H}_{n}$. 
Unlike the customary approach, we consider dual frame $G$ to encode a signal and $F$ to decode the measurements.
 Then every erasure of encoding frame coefficients as $\{\langle f,g_{i}\rangle \}_{i\in \sigma}$, $\vert \sigma\vert\leq spark F-1$ can be exactly recovered by Corollary \ref{frameco}. So,  frames with large  spark are  resilient  against more erasures of associated dual frame coefficients; since for frames $K=I_{\mathcal{H}_{n}}$ we call $(r,k)$-matrix associated to $F$ and $G$ an $r$-matrix for convenience. In this case,   if $M_{F,G}$ is an $r$-matrix and  $N$ is an $r$-erasures recovery matrix for $F$, i.e., $NF^{*}=0$ and $spark N= r+1$ then
\begin{eqnarray*}
NM_{F,G}G^{*}=NF^{*}=0.
\end{eqnarray*} 
Thus, $NM_{F,G}$ is an $\rho$-erasure recovery matrix for $G$ with $\rho=sparkNM_{F,G}\geq r+1$ since $Ker M_{F,G} \subseteq Ker NM_{F,G}$.
\subsection{The existence of $(r,k)$-matrices }
 In the following, we show the relation between the existence of $(r,k)$-matrices with MRC.
The following result gives a  necessary condition for the existence of $(r,k)$-matrices and a sufficient condition for a $K$-dual to satisfy MRC.
 \begin{thm}\label{Mrcresult}
Suppose that $F=\{f_{i}\}_{i\in I_{m}}$ is a   $K$-frame of $\mathcal{H}_{n}$ and $G\in KD_{F}$. If  there exists a matrix $M_{F,G}$ which satisfies $(\ref{r,kmatrix})$. Then $G$ satisfies MRC for $(spark(M_{F,G}-\mathcal{G}_{F})-1)$-erasures.
\end{thm}
\begin{proof}
If $M_{F,G}$ has a zero  column then $spark(M_{F,G}-\mathcal{G}_{F})=1$ so the result clearly holds. Now, let all columns of $M_{F,G}$ are non-zero. Hence, $spark(M_{F,G}-\mathcal{G}_{F})\geq 2$. 
Also, by the assumption,
we have 
$(M_{F,G}-\mathcal{G}_{F})G^{*}=0.$
 On the other hand, if all columns of $M_{F,G}-\mathcal{G}_{F}$ are linearly independent, then it is invertible which implies $\theta_{G}^{*}=0$ that is a contradiction. Therefore, 
\begin{eqnarray*}
2\leq spark(M_{F,G}-\mathcal{G}_{F}) < \infty.
\end{eqnarray*}
Now, considering
\begin{eqnarray*}
\rho=spark(M_{F,G}-\mathcal{G}_{F})-1
\end{eqnarray*}
implies that every $\rho$ columns of $M_{F,G}-\mathcal{G}_{F}$ is linearly independent and so
  every $Kf\in R(K)$ can be recovered from the coefficients $\{\langle f,g_{i}\rangle \}_{i\in \Lambda^{c}}$  for every $\Lambda \subset I_{m}$, $\vert \Lambda\vert\leq \rho$. Without loss of the generality, we discuss  the first $\rho$ columns. More precisely, for every $f\in \mathcal{H}_{n}$ there exists $\{\alpha_{i,j}\}_{j=\rho+1}^{m}$ so that
\begin{eqnarray*}
\langle f, g_{i}\rangle =\sum_{j=\rho+1}^{m}\alpha_{i,j}\langle f, g_{j}\rangle, \quad (i\in I_{\rho}).
\end{eqnarray*}
Thus, we can write
\begin{eqnarray*}
\Vert Kf\Vert^{4}&=& \left\vert  \langle Kf, Kf\rangle\right\vert^{2}\\
&=&\left\vert  \left\langle \sum_{i=1}^{\rho}\sum_{j=\rho+1}^{m}\alpha_{i,j}\left\langle f, g_{j}\right\rangle f_{i}+\sum_{i=\rho+1}^{m}\left\langle f, g_{i}\right\rangle f_{i}, Kf \right\rangle\right\vert^{2}\\
&=&\left\vert \left\langle \sum_{j=\rho+1}^{m}\langle f, g_{j}\rangle  \left(f_{j}+ \sum_{i=1}^{\rho} \alpha_{i,j}f_{i}\right), Kf \right\rangle\right\vert^{2}\\
&\leq& \sum_{j=\rho+1}^{m}\left\vert \langle f, g_{j}\rangle  \right\vert^{2}\sum_{j=\rho+1}^{m}\left\vert \left\langle \left(f_{j}+ \sum_{i=1}^{\rho} \alpha_{i,j}f_{i}\right), Kf \right\rangle \right\vert^{2}\\
&\leq& B\Vert Kf\Vert^{2}\sum_{j=\rho+1}^{m}\left\vert \langle f, g_{j}\rangle  \right\vert^{2},
\end{eqnarray*}
where the existence of the upper bound $B$ in  the last inequality is assured by the assumption that $F$ is a $K$-frame. Therefore,
\begin{eqnarray*}
\Vert Kf\Vert^{2}/B\leq\sum_{j=\rho+1}^{m}\left\vert \langle f, g_{j}\rangle  \right\vert^{2}\leq D\Vert f\Vert^{2},
\end{eqnarray*}
for an upper bound $D$ of $G$. Hence, $\{g_{i}\}_{i=\rho+1}^{m}$ is also a $K^{*}$-frame for $\mathcal{H}_{n}$. This follows the desired result.
\end{proof}

\section{Examples}
In this section, we present several  examples to show not only the importance of the necessary or  sufficient conditions in our main results, but also the  adventages and  significant differences of $(r,k)$-matrices with respect to $r$-erasure recovery matrices \cite{han2020}. In particular, some examples of   some $K$-frames  (frames) is given for which there does not exist any  appropriate erasure recovery matrix, but infinitely many $(r,k)$-matrices.
 In this section, we consider  $\lbrace e_{i}\rbrace_{i\in I_{n}}$ as  the  standard orthonormal basis of $\mathbb{R}^{n}$.
The first example shows that
the converse of Theorem \ref{1thm} does not hold in general.
\begin{ex}\label{1ex1}
Consider, $\mathcal{H}= \mathbb{R}^{3}$ and $F=\{e_{1}, e_{2}\}$. Also, let $K\in B(\mathcal{H} )$ so that $Kf=\left(c_{1}+c_{2}+\dfrac{1}{2}c_{3}\right)e_{1}$, for every $f=\sum_{i\in I_{3}}c_{i}e_{i}$.
Then $F$ is a $K$-frame for $\mathcal{H}$ and $S_{F}\vert_{R(K)}=I_{R(K)}$. Hence
\begin{eqnarray*}
G:=K^{*}(S_{F}\vert_{R(K)})^{-1}\pi_{S_{F}(R(K))}F=K^{*}((S_{F}\vert_{R(K)})^{-1})^{*}\pi_{R(K)}F=\{e_{1}+e_{2}+\dfrac{1}{2}e_{3}, 0\}.
\end{eqnarray*}
A straightforward computation reveals that
\begin{equation*}
\theta_{F}\theta_{G}^{*} = K,
\end{equation*}
so $G\in KD_{F}$. Moreover $S_{G}=\theta_{G}\theta_{Z}^{*}$ for every $K$-dual $Z$ of $F$, i.e., $G$ is the canonical $K$-dual of $F$, however $F$ is not a subset of  $R(K)$ or $S_{F}(R(K))$.
\end{ex}
Also, the following example shows the importance of the sufficient conditions in Theorem \ref{1thm}.
\begin{ex}\label{ex1}
Let $\mathcal{H} = \mathbb{R}^{4}$ and define
\begin{equation*}
{F}=\left\{ \left[
 \begin{array}{ccc}

1\\
\\
0 \\
\\
0 \\
\\
0\\
\end{array} \right], \left[
 \begin{array}{ccc}

0\\
\\
1\\
\\
0\\
\\
0\\
\end{array} \right], \left[
 \begin{array}{ccc}

0\\
\\
0\\
\\
1\\
\\
0\\
\end{array} \right], \left[
 \begin{array}{ccc}

 1\\
\\
0\\
\\
1\\
\\
0\\
\end{array} \right] \right\},
\end{equation*}
and $K\in B(\mathcal{H} )$ as $Kf=(c_{1}+c_{3})e_{1}+(c_{2}+\dfrac{1}{2}c_{4})e_{2}$, for every $f=\sum_{i\in I_{4}}c_{i}e_{i}$.
Then $F$ is a $K$-frame for  $\mathcal{H}$ and
\begin{eqnarray*}
\pi_{R(K)}F=\{e_{1},e_{2},0,e_{1}\}.
\end{eqnarray*}
Furthermore, the restriction of $K$-frame operator
\begin{eqnarray*}
S_{F}\vert_{R(K)}:\textit{span}{\{e_{1},e_{2}\}}\rightarrow \textit{span}{\{2e_{1}+e_{3},e_{2}\}}
\end{eqnarray*}
 is given by
\begin{eqnarray*}
S_{F}\vert_{R(K)}(e_{i}) =\begin{cases}
\begin{array}{ccc}
2e_{1}+e_{3}& \;
{i=1}, \\
e_{2}& \; {i=2} \\
\end{array}
\end{cases}
\end{eqnarray*}
that  is an  invertible operator. So,  we obtain
\begin{equation*}
{G:=K^{*}(S_{F}\vert_{R(K)})^{-1}\pi_{S_{F}(R(K))}F}=\left\{ \left[
 \begin{array}{ccc}

\dfrac{2}{5}\\
\\
0 \\
\\
\dfrac{2}{5}  \\
\\
0\\
\end{array} \right], \left[
 \begin{array}{ccc}

0\\
\\
1\\
\\
0\\
\\
\dfrac{1}{2}\\
\end{array} \right], \left[
 \begin{array}{ccc}

\dfrac{1}{5}\\
\\
0\\
\\
\dfrac{1}{5}\\
\\
0\\
\end{array} \right], \left[
 \begin{array}{ccc}

 \dfrac{3}{5}\\
\\
0\\
\\
 \dfrac{3}{5}\\
\\
0\\
\end{array} \right] \right\}.
\end{equation*}
The Bessel sequence $G$ is a $K$-dual of $\pi_{R(K)}F$. However, $G$ is neither a $K$-dual of $F$, nor the canonical $K$-dual of $\pi_{R(K)}F$. Indeed, for every $f=\sum_{i\in I_{4}}c_{i}e_{i}$
\begin{equation*}
\theta_{F}\theta_{G}^{*}f=e_{1}(c_{1}+c_{3})+e_{2}(c_{2}+c_{4}/2)+4/5 e_{3}(c_{1}+c_{3}),
\end{equation*}
and consequently,
\begin{equation*}
\theta_{F}\theta_{G}^{*}e_{3} = e_{1}+\dfrac{4}{5}e_{3}  \neq Ke_{3}.
\end{equation*}
Thus, $G$ is not a $K$-dual of $F$.
Also, consider
\begin{equation*}
H= \left\{g_{1}, g_{2}, 0, g_{4}\right \}.
\end{equation*}
Then $H$ is a $K$-dual of $\pi_{R(K)}F$. Moreover,
 $\Vert  \theta_{H}^{*}\Vert <\Vert \theta_{G}^{*}\Vert$ and this implies that $G$ cannot be the canonical $K$-dual of $\pi_{R(K)}F$.
\end{ex}
The next example shows that the condition $(i)$ in Theorem \ref{fthm} is not  sufficient for a subset $\sigma\subset I_{m}$ to satisfies MRC.
\begin{ex}\label{4.1ex}
Let $\mathcal{H}$, $K$ and $F$ be as in Example \ref{ex1}. Take $\sigma=\{1,3\}$ then the sequence  $\{f_{i}\}_{i\in \sigma^{c}}=\{e_{2}, e_{1}+e_{3}\}$ clearly is not a $K$-frame and so $\sigma$ does not satisfy MRC. However,
\begin{eqnarray*}
R(\theta_{F}^{*}K)= \textit{span}\{(a, b,0,a): a,b\in \mathbb{R}\},
\end{eqnarray*}
which implies that $R(\theta_{F}^{*}K)\cap \textit{span}\{\delta_{i}\}_{i\in \sigma}=\{0\}$.
\end{ex}
In the sequel, we observe the advantages of using  $(r,k)$-matrices with respect to $r$-erasure recovery matrices.
In fact,   we present some $K$-frames  (frames)   for which there does not exist any  appropriate erasure recovery matrix, but infinitely many $(r,k)$-matrices.

\begin{ex}
Suppose that 
\begin{equation*}
{F}=\left\{ \left[
 \begin{array}{ccc}

1\\
\\
0 \\
\\
-1 \\
\\
0\\
\end{array} \right], \left[
 \begin{array}{ccc}

0\\
\\
0\\
\\
1\\
\\
0\\
\end{array} \right], \left[
 \begin{array}{ccc}

0\\
\\
0\\
\\
-1\\
\\
2\\
\end{array} \right], \left[
 \begin{array}{ccc}

 1/2\\
\\
0\\
\\
1/2\\
\\
0\\
\end{array} \right] \right\},
\end{equation*}
and $K\in B(\mathbb{R}^{4})$ is defined by $Ke_{1}=Ke_{2}=Ke_{3}=e_{1}$ , $Ke_{4}=e_{4}-e_{1}$. Then $F$ is a $K$-frame for $\mathbb{R}^{4}$ and the Gramian matrix is obtained by
\begin{equation*}
\mathcal{G}_{F} = \left[
 \begin{array}{ccc}

2\quad -1 \quad  1 \quad 0\\

-1\quad 1 \quad -1 \quad ‎\dfrac{1}{2} \\

1\quad -1 \quad 5 \quad ‎\dfrac{-1}{2} \\

0\quad ‎\dfrac{1}{2}  \quad ‎\dfrac{-1}{2} \quad ‎\dfrac{1}{2}‎ \\

\end{array} \right].
\end{equation*}
Since $spark \mathcal{G}_{F}=3$ by Theorem \ref{Matrix eq} (ii),  we can derive any $2$-erasures   of  $K$-dual frame coefficients for every $K$-dual of $F$. Moreover, there is not any appropriate erasure recovery matrix for $F$. Indeed, if $N$ is an $l\times 4$ matrix so that $NF^{*}=0$ then the third column of $N$ is zero, i.e., $spark N=1$. Thus $N$ can only preserves one erasure of $K$-frame coefficients. However, there exist infinitely many $(r,k)$-matrix. Put 
\begin{equation*}
{G}=\left\{ \left[
 \begin{array}{ccc}

1\\
\\
1 \\
\\
1 \\
\\
1/2\\
\end{array} \right], \left[
 \begin{array}{ccc}

1\\
\\
1\\
\\
1\\
\\
1/2\\
\end{array} \right], \left[
 \begin{array}{ccc}

0\\
\\
0\\
\\
0\\
\\
1/2\\
\end{array} \right], \left[
 \begin{array}{ccc}

 0\\
\\
0\\
\\
0\\
\\
1\\
\end{array} \right] \right\}.
\end{equation*}
One may check that $FG^{*}=K$ so the sequence $G$ is a $K$-dual of $F$ and every matrix $M_{F,G}$ satisfies $(\ref{r,kmatrix})$ is as follows
\begin{equation*}
M_{F,G} = \left[
 \begin{array}{ccc}

‎\alpha_{1}\quad 1-‎\alpha_{1} \quad  ‎\alpha_{2} \quad ‎\dfrac{1}{2}‎-‎\dfrac{1}{2}‎‎\alpha_{2}\\

‎\beta_{1}\quad -‎\beta_{1} \quad ‎\beta_{2} \quad -‎\dfrac{1}{2}‎\beta_{2} \\

‎‎\gamma_{1}\quad -‎‎\gamma_{1} \quad ‎‎\gamma_{2} \quad 2-‎\dfrac{1}{2}‎‎\gamma_{2} \\

\eta‎‎‎‎_{1}\quad ‎\dfrac{1}{2}-\eta‎‎‎‎_{1}  \quad \eta‎‎‎‎_{2} \quad ‎\dfrac{1}{4}-‎\dfrac{1}{2}\eta‎‎‎‎_{2} \\

\end{array} \right]
\end{equation*}
Therefore, we can find infinitely many $(r,k)$-matrix associated with $F$ and $G$, for $1\leq r \leq 3$. For example set $‎\alpha_{1}=‎\alpha_{2}=1, ‎\beta_{1}=‎\beta_{2}=1, ‎‎\gamma_{1}=-1, ‎‎\gamma_{2}=2$ and $\eta‎‎‎‎_{1}=\eta‎‎‎‎_{2}=1/2$ we obtain an $(3,k)$-matrix.





\end{ex}
In the last example, we survey the case that a $K$-frame $F$ satisfies $F\subseteq R(K)$, i.e., it can be considered as a frame for $R(K)$.
Moreover, we observe that unlike $r$-erasure recovery matrices  for ordinary frames \cite{han2020}   the existence of $(r,k)$-matrices is independent of the fact that a $K$-frame or its dual satisfies MRC.
\begin{ex}\label{finexam}
Consider 
\begin{equation*}
{F}=\left\{ \left[
 \begin{array}{ccc}

0\\
\\
0 \\
\\
0 \\
\\
-1\\
\end{array} \right], \left[
 \begin{array}{ccc}

0\\
\\
1\\
\\
0\\
\\
0\\
\end{array} \right], \left[
 \begin{array}{ccc}

0\\
\\
2\\
\\
0\\
\\
-1\\
\end{array} \right], \left[
 \begin{array}{ccc}

 1\\
\\
0\\
\\
0\\
\\
0\\
\end{array} \right] \right\},
\end{equation*}
Also let $Kf= c_{1}e_{1}+c_{2}e_{2}+(c_{3}+c_{4})e_{4}$, for every $f=\sum_{i\in I_{4}}c_{i}e_{i}$.
Then $F$ is a $K$-frame for  $\mathbb{R}^{4}$ and
\begin{equation*}
\mathcal{G}_{F} = \left[
 \begin{array}{ccc}

1\quad 0 \quad  1 \quad 0\\

0\quad 1 \quad 2 \quad 0 \\

1\quad 2 \quad 5 \quad 0 \\

0\quad 0  \quad 0 \quad 1 \\

\end{array} \right].
\end{equation*}
Thus $spark \mathcal{G}_{F} =3$ and we can consider $\mathcal{G}_{F}$ as a $(2,k)$-matrix associated with $F$ and each one of its $K$-duals. We observe that, non of $2$-columns in $F$ produce $R(K)$. Put
\begin{equation*}
{G}=\left\{ \left[
 \begin{array}{ccc}

0\\
\\
0 \\
\\
-1 \\
\\
-1\\
\end{array} \right], \left[
 \begin{array}{ccc}

0\\
\\
1\\
\\
0\\
\\
0\\
\end{array} \right], \left[
 \begin{array}{ccc}

0\\
\\
0\\
\\
0\\
\\
0\\
\end{array} \right], \left[
 \begin{array}{ccc}

 1\\
\\
0\\
\\
0\\
\\
0\\
\end{array} \right] \right\},
\end{equation*}
we have that
\begin{equation*}
\theta_{F}\theta^{*}_{G}=K,
\end{equation*}
i.e., $G\in KD_{F}$. However, $R(K^{*})=\textit{span}\{e_{1}, e_{2}, e_{3}+e_{4}\}$ and so $G$ does not satisfy MRC even for $1$-erasures. Moreover, non of $2$-columns in $G$ remain $K^{*}$-frame for  $\mathbb{R}^{4}$. It is worth to note that by taking
\begin{equation*}
M_{F,G} = \left[
 \begin{array}{ccc}

1\quad 0 \quad  a_{1} \quad 0\\

0\quad 1 \quad a_{2} \quad 0 \\

1\quad 2 \quad a_{3} \quad 0 \\

0\quad 0  \quad a_{4} \quad 1 \\

\end{array} \right],
\end{equation*}
for all $a_{i}\in \mathbb{R}$, $i\in I_{4}$, we obtain   a family of $(r,k)$-matrices with respect to $F$ and $G$ which $r$ is dependent on the choice of $a_{i}$, $i\in I_{4}$. In this case  $3$-columns of $M_{F,G}-\mathcal{G}_{F}$ are zero but with appropriate choices of $a_{i}$, we get $spark M_{F,G}=4$.
\end{ex}
It is worth noticing that $F$ in Example \ref{finexam} is also a frame for $R(K)$. From this point of view every $m\times 4$  matrix $N$ so that  $NF^{*}=0$ has a zero column. Hence $spark N=1$ and so there is no appropriate
 erasure recovery matrix  for $F$, however we obtain infinitely  $(r,k)$-matrices with respect to $F$ and $G$. 



\bibliographystyle{amsplain}

\end{document}